\documentclass[12pt,reqno]{amsart}
\usepackage[utf8]{inputenc}
\usepackage{graphicx}
\usepackage{amsfonts}
\usepackage{amssymb}
\usepackage[centertags]{amsmath}
\usepackage{anysize} 
\usepackage{float}
\usepackage{subfigure}
\usepackage{multirow}
\usepackage{enumerate}
\usepackage{cite}
\usepackage{hyperref}
\usepackage{orcidlink}

\hypersetup{colorlinks=true,linkcolor=blue,citecolor=red}

\def\sin{{\,\rm sen \, }} 
\DeclareGraphicsExtensions{.pdf,.eps,.png,.jpg}

\newtheorem{proposition}{Proposition}

\theoremstyle{definition}

\begin{document}

\title[On distinguishing genuine from spurious chao]{On distinguishing genuine from spurious chaos in planar singular and nonsmooth systems: A diagnostic approach}

\author[Martha  Alvarez]{Martha Alvarez-Ram\'{\i}rez \orcidlink{0000-0001-9187-1757}}

\address{Departamento de Matem\'aticas, UAM--Iztapalapa, 09310 Iztapalapa, Mexico City, Mexico}
\email{mar@xanum.uam.mx}

\begin{abstract}
We present a rigorous reassessment of chaotic behavior in two-dimensional
autonomous systems with singular or nonsmooth dynamics. For the Cummings-Dixon-Kaus
(CDK) model, we show that blow-up regularization restores smoothness and
renders the hypotheses of the Poincar\'e-Bendixson theorem applicable, thereby
excluding chaotic attractors away from the singular set. We prove topological
equivalence between the original and regularized flows on annular domains,
ensuring that no spurious invariant sets are introduced by desingularization.
In contrast, for a nonsmooth system with a $|x|$ term, we recompute the entire
period-doubling cascade, obtain a seven-term sequence of bifurcation values
converging to Feigenbaum's constant, and confirm robust chaos through positive
Lyapunov exponents, broadband spectra, and fractal dimension estimates. As a
main outcome, we propose a diagnostic protocol integrating regularization,
numerical refinement, and invariant-set criteria. This protocol provides a
reproducible standard for distinguishing genuine planar chaos from artifacts
caused by singularities or discretization, and offers a benchmark for future
studies of low-dimensional nonsmooth systems.
\end{abstract}

\subjclass[2020]{34C28, 34A36, 37D45, 37G35}
\keywords{Hopf bifurcation, business cycles, dynamical systems, Keen model, center manifold}

\maketitle

\section{Introduction}\label{sec1}
The Poincaré-Bendixson theorem sets a classical boundary for dynamical systems: smooth, autonomous
flows in the plane cannot host chaotic attractors. Under this framework, genuine chaos is expected only
in higher-dimensional systems, in two-dimensional nonautonomous models, or in flows with explicit
nonsmoothness. This result has shaped our understanding of low-dimensional dynamics for decades.

Yet, recent contributions have challenged this perspective by reporting chaos-like behavior in planar
systems. These claims often arise in contexts where the assumptions of the Poincaré-Bendixson theorem do
not apply, including singularities, loss of uniqueness, or deliberate nonsmooth terms. For example, irregular
dynamics in singular models have been interpreted as chaos \cite{Alvarez2005}, while subsequent
regularization revealed that such behavior disappears \cite{Seiler2021}. In contrast, nonsmooth constructions
can bypass the theorem entirely and support robust chaotic attractors \cite{de2022chaotic}. More recently,
numerical simulations of physical systems with singular solutions have been presented as evidence of chaos
\cite{Buts2024}. Although distinct in nature, these cases are frequently conflated. This conflation blurs the
boundary between genuine chaos and spurious complexity.

While earlier works examined these cases in isolation, our contribution is to unify them within a
common diagnostic framework, providing a coherent and systematic comparison that offers
conceptual and methodological clarity. This paper contrasts singular, regularized, and nonsmooth planar systems under a unified perspective and formulates diagnostic criteria that distinguish numerical artifacts from mathematically valid chaos.
This approach clarifies long-standing debates on the scope of the
Poincaré-Bendixson theorem and elucidates the role of singularities, loss of
Lipschitz continuity, and nonsmooth dynamics in generating apparent
low-dimensional chaos. Beyond a critical synthesis of prior studies, the
paper contributes new numerical results: we recompute the bifurcation diagram
of a representative nonsmooth system over the full parameter range where the
period-doubling cascade occurs, and report a detailed sequence of bifurcation
values $\{a_n\}$ confirming Feigenbaum's universality. These computations
corroborate the proposed diagnostic criteria and provide a reproducible
benchmark for future investigations of chaos in planar nonsmooth systems.

From the comparative analysis, a set of practical criteria emerges for
discriminating genuine chaos from numerical or modeling artifacts. Chaotic
signatures must be robust under numerical refinement (smaller time steps and
higher precision) and under small parameter perturbations; when singularities
are present, blow-up or regularization should be applied to verify whether the
attractor persists in the desingularized system; and the resulting dynamics
should display invariant-set evidence consistent with chaos, such as positive
Lyapunov exponents, dense periodic orbits, positive entropy, or broadband
spectra. Together, these principles define the diagnostic standard adopted in
Sections~\ref{sec2} and~\ref{sec3} and provide a reproducible reference for
future investigations.

Following these principles, we present four main results. First, we
formalize a diagnostic checklist that operationalizes when planar ``chaos''
should be regarded as genuine rather than a singularity- or
discretization-induced artifact. Second, for the CDK model we clarify why
regularization restores the hypotheses of the Poincaré-Bendixson theorem and
establish topological equivalence (away from the singular set) to rule out
spurious dynamics introduced by desingularization. Third, for a representative
nonsmooth system we perform new computations: we obtain the complete
bifurcation diagram across the entire period-doubling window and report a
seven-term sequence $\{a_n\}$ with Feigenbaum scaling, thereby providing a
reference dataset for reproducibility. Finally, we discuss the modeling
implications of singular systems and propose a practical protocol that
harmonizes mathematical rigor with engineering evidence. This diagnostic
protocol may be directly applicable to engineering systems where singularities
or dry-friction terms generate apparent chaos, such as vibro-impact oscillators,
electronic switching circuits, and models of astrophysical bursts.

The remainder of the paper is organized as follows.
Section~\ref{sec2} revisits the Cummings-Dixon-Kaus model, highlighting the
apparent chaotic dynamics and showing how they disappear under rigorous
regularization. Section~\ref{sec3} analyzes the nonsmooth construction of
Zain-Aldeen et al., emphasizing its bifurcation mechanisms and numerical
evidence for robust chaos. Section~\ref{sec3.1} addresses the interpretation
of chaos arising from singular solutions in physical systems, as proposed by
Buts and collaborators. Finally, Section~\ref{sec4} summarizes the main
conclusions and discusses the broader implications of distinguishing spurious
complexity from genuine chaotic dynamics in planar flows.

\section{Apparent chaos and regularization in the Cummings-Dixon-Kaus model: From heuristics to rigorous analysis}\label{sec2}
Beyond reviewing the Cummings-Dixon-Kaus (CDK) model and its historical
interpretations, this section highlights two original contributions. First,
we clarify how the regularization restores the hypotheses of the
Poincaré-Bendixson theorem and rigorously excludes chaos away from the
singularity. Second, we formulate explicit conditions for the topological
equivalence between the original and regularized flows on the nonsingular
domain, ensuring that no spurious dynamics are introduced by the
desingularization. Together, these results provide a reproducible reference
for future studies of singular planar systems and set the stage for a
re-examination of the original CDK formulation in its astrophysical context.

A dynamical model for the evolution of the magnetic field in isolated neutron
stars was proposed by Cummings, Dixon, and Kaus \cite{Cummings1992}, hereafter
referred to as the CDK model. Their model captures three distinct regimes
bursters, pulsars, and magnetically aligned ``dead'' stars, emerging from the
nonlinear interaction between differential rotation and magnetic damping. In
particular, the erratic behavior observed in gamma-ray bursters is explained as
a result of repeated transitions through a singular configuration where the
magnetization vanishes. The model predicts that neutron stars may follow
different evolutionary paths, typically transitioning from an initial chaotic
burster phase to a pulsar state, and eventually to a non-radiating, magnetically
aligned state. This framework provides a unified dynamical interpretation of the
observed diversity among neutron stars and offers a plausible mechanism for the
generation of short, intense gamma-ray bursts via magnetospheric processes.

The core dynamics of the system are captured by a set of nonlinear ordinary differential equations governing the evolution of the magnetization vector $\vec{m} = (m_x, m_y, m_z)$ in dimensionless form:
\begin{align}\label{eq_chaos1}
\dot{m}_x &= m_y - \frac{\lambda m_x m_z}{m_x^2 + m_y^2 + m_z^2} - \epsilon m_x, \nonumber \\
\dot{m}_y &= -m_x - \frac{\lambda m_y m_z}{m_x^2 + m_y^2 + m_z^2} - \epsilon m_y, \\
\dot{m}_z &= \frac{\lambda m_z^2}{m_x^2 + m_y^2 + m_z^2} - \overline{\epsilon} m_z - (\lambda - \overline{\epsilon})\nonumber ,
\end{align}
where  $m_x, m_y, m_z$ are the Cartesian components of magnetization (dimensionless), $\lambda$ is the scaled Landau damping parameter (unitless), $\epsilon$ and $\overline{\epsilon}$ are the scaled mechanical damping coefficients.
Cummings et al. observed that, in a specific region of parameter space  where $\epsilon/\lambda < 1$ and $\overline{\epsilon}/\lambda < 1$, the magnetization, and hence the magnetic field, exhibits erratic, nonperiodic behavior, occasionally generating intense pulses of directional electric and magnetic fields. This regime corresponds to repeated attraction of the system toward a singularity at the origin, resulting in unpredictable bursts of magnetic activity. Such behavior is associated with the short-duration gamma-ray emissions characteristic of bursters.

The unpredictability of the orbits arises from their extreme sensitivity to initial conditions near the origin, a characteristic that led Cummings et al.~to suspect chaotic behavior. They also noted that system~\eqref{eq_chaos1} exhibits axial symmetry about the $m_z$-axis, which motivated the adoption of cylindrical coordinates through the transformation $(m_x, m_y, m_z) \mapsto (x\cos\phi, x\sin\phi, z)$. Under this change of variables, the system takes the form:
\begin{equation}
\begin{aligned} \label{eq_chaos2d}
\dot{x} &= \frac{x z}{x^2 + z^2} - a x, \\
\dot{z} &= \frac{z^2}{x^2 + z^2} - b z + b - 1,\\
\dot{\phi} &=  1.
\end{aligned}
\end{equation}
where $a= \epsilon/\lambda$ y $b= \overline{\epsilon}/\lambda$. 
The dynamics of the reduced $(x, z)$ subsystem can thus be analyzed independently, as it decouples from the angular variable $\phi$ in the full three-dimensional model~\eqref{eq_chaos1}. 

In \cite{Alvarez2005},  Alvarez-Ram\'{\i}rez  et al. report that numerical simulations of both the full three-dimensional model and its planar reduction show erratic behavior and pronounced sensitivity to initial conditions. They demonstrate that trajectories can undergo abrupt deflections when passing near the singularity at the origin, even under arbitrarily small perturbations. The resulting time series and phase portraits, exemplified in Figure \ref{fig:pupu}, capture the system’s unpredictable dynamics and strong sensitivity to local geometric structures.
To investigate the nature of this irregular behavior, the authors employ a regularization method based on polar blow-up techniques. This analysis shows that the system contains an infinite family of homothetic orbits., all emanating from and returning to the so-called collision manifold $\Lambda = { r = 0 }$. This geometric structure leads to recurrent transitions between ejection and collision trajectories. The authors conclude the resulting irregular dynamics as a form of piecewise deterministic behavior, where small perturbations near the singularity give rise to abrupt shifts between distinct homothetic orbits.
\begin{figure}[hbpt] 
\centering
\subfigure[$(m_x,m_y)$-plane]
{ \includegraphics[scale=0.08]{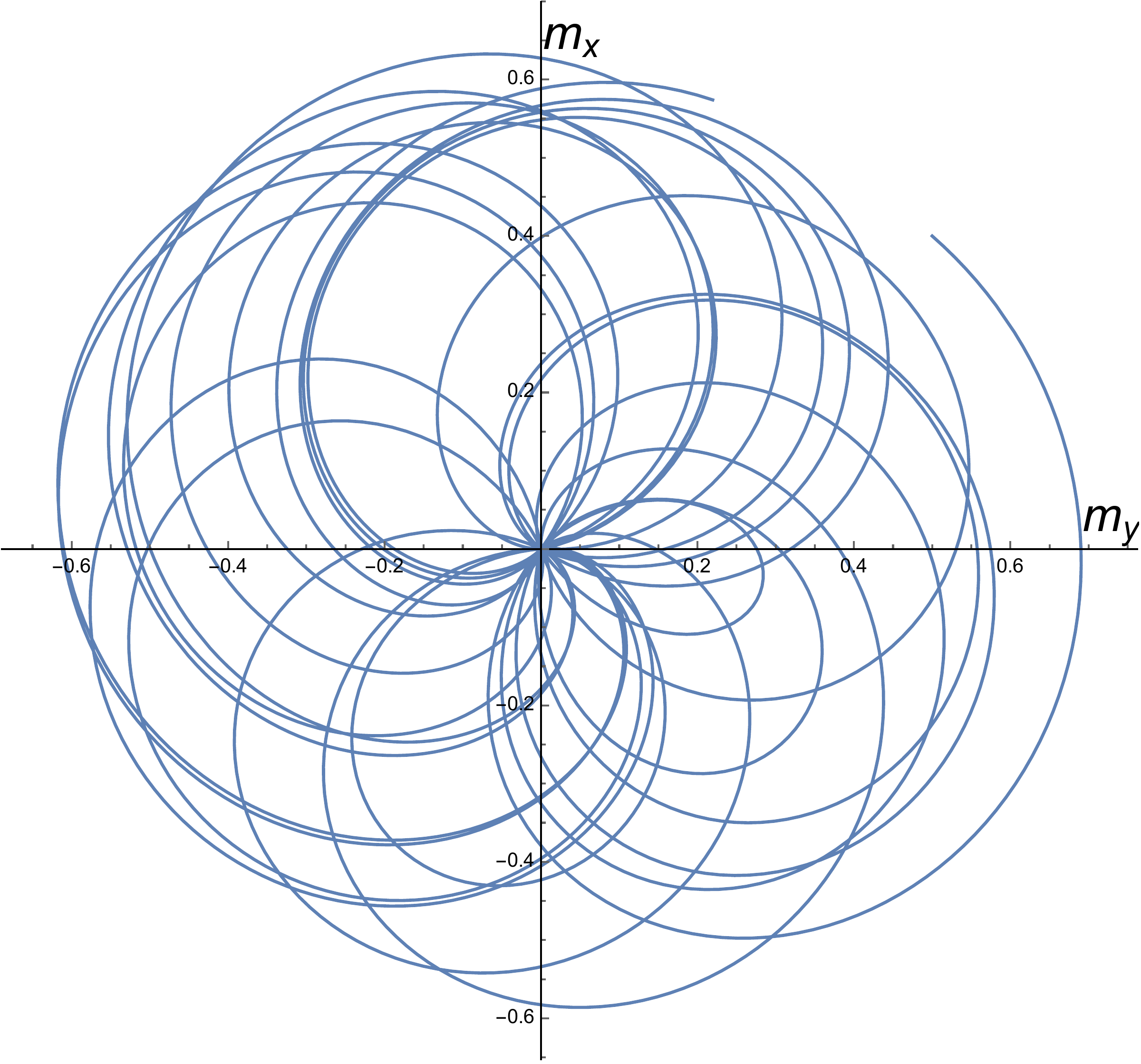}
}\quad
\subfigure[$(m_x,m_z)$-plane]
{ \includegraphics[scale=0.08]{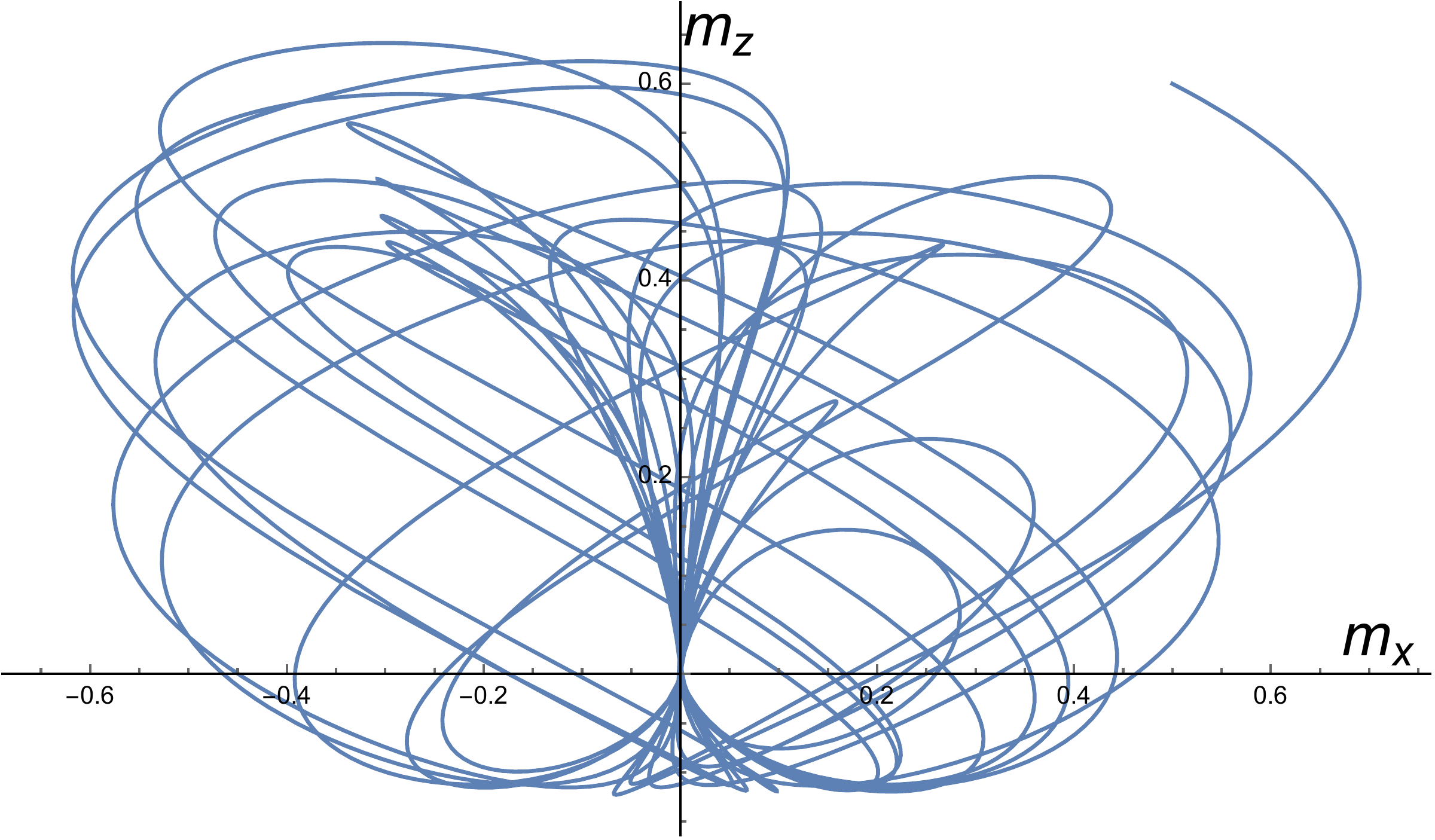}
}\quad
\subfigure[$(m_z,m_y)$-plane]
{ \includegraphics[scale=0.08]{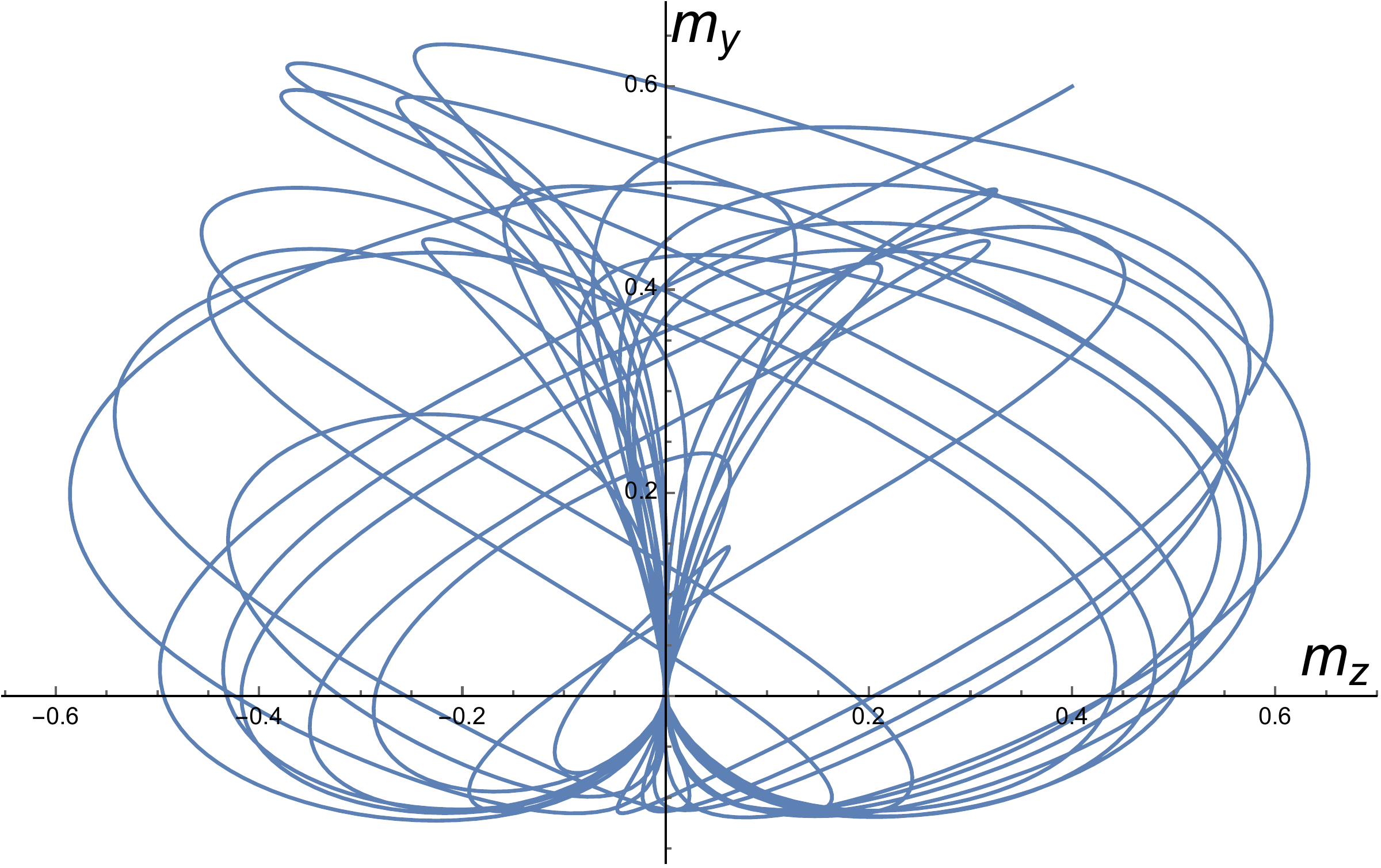}}\qquad
\subfigure[Time series of the magnetization components: $m_x$ (blue), $m_y$ (yellow), and $m_z$ (green)]
{ \includegraphics[scale=0.39]{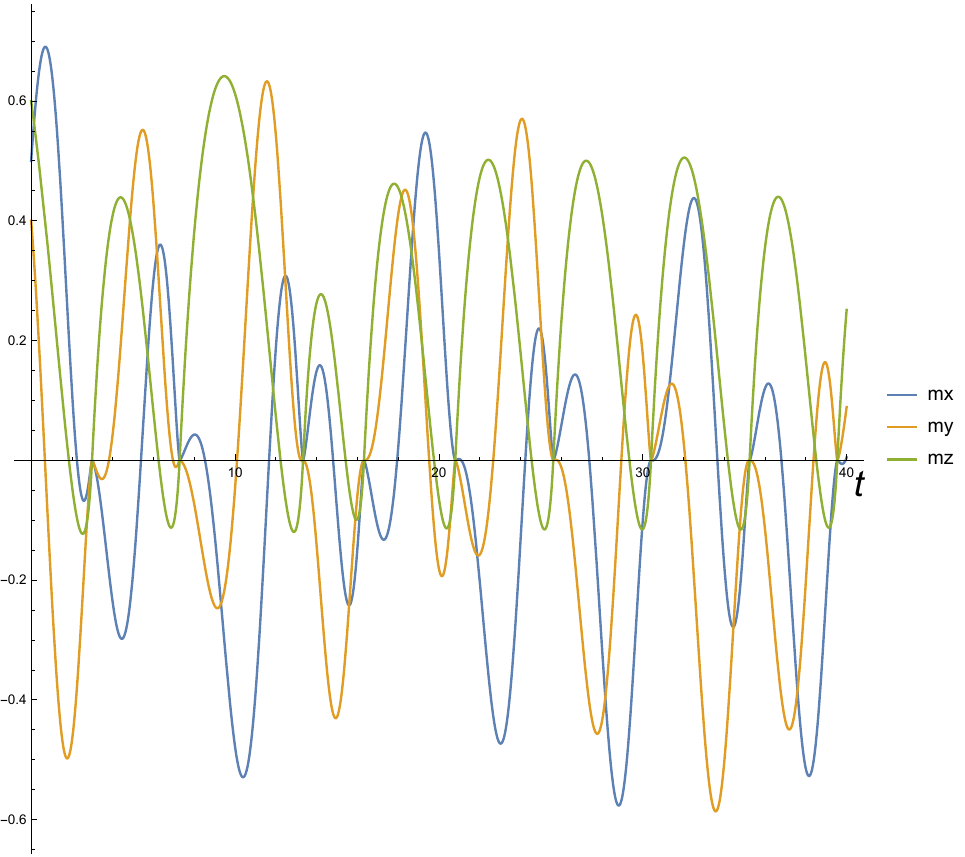}}\hfill
\caption{Phase portraits of the three-dimensional CDK system for representative parameter ratios. The projections onto different coordinate planes illustrate the irregular dynamics and sensitivity to initial conditions near the singularity, while the time series highlights the erratic temporal behavior.}
    \label{fig:pupu}
\end{figure}

In addition, Alvarez-Ramírez et al. conducted numerical simulations of the reduced system~\eqref{eq_chaos2d}, obtaining phase portraits of the type shown in Figure~\ref{fig:pupu2}. Based on these numerical experiments, the authors invoke the Poincaré-Bendixson theorem to suggest the existence of limit cycles in the reduced system.  However, their application of the theorem lacks full mathematical justification. Although a polar blow-up is employed to regularize the singularity at the origin, the resulting vector field remains non-smooth and is not defined at that point. As a result, the system does not satisfy the continuity and differentiability conditions required for a rigorous application of the theorem. Their argument therefore offers a compelling heuristic interpretation of the observed recurrent dynamics, but ultimately falls short of providing a definitive mathematical conclusion.
\begin{figure}[hbpt]
\centering
\includegraphics[scale=0.09]{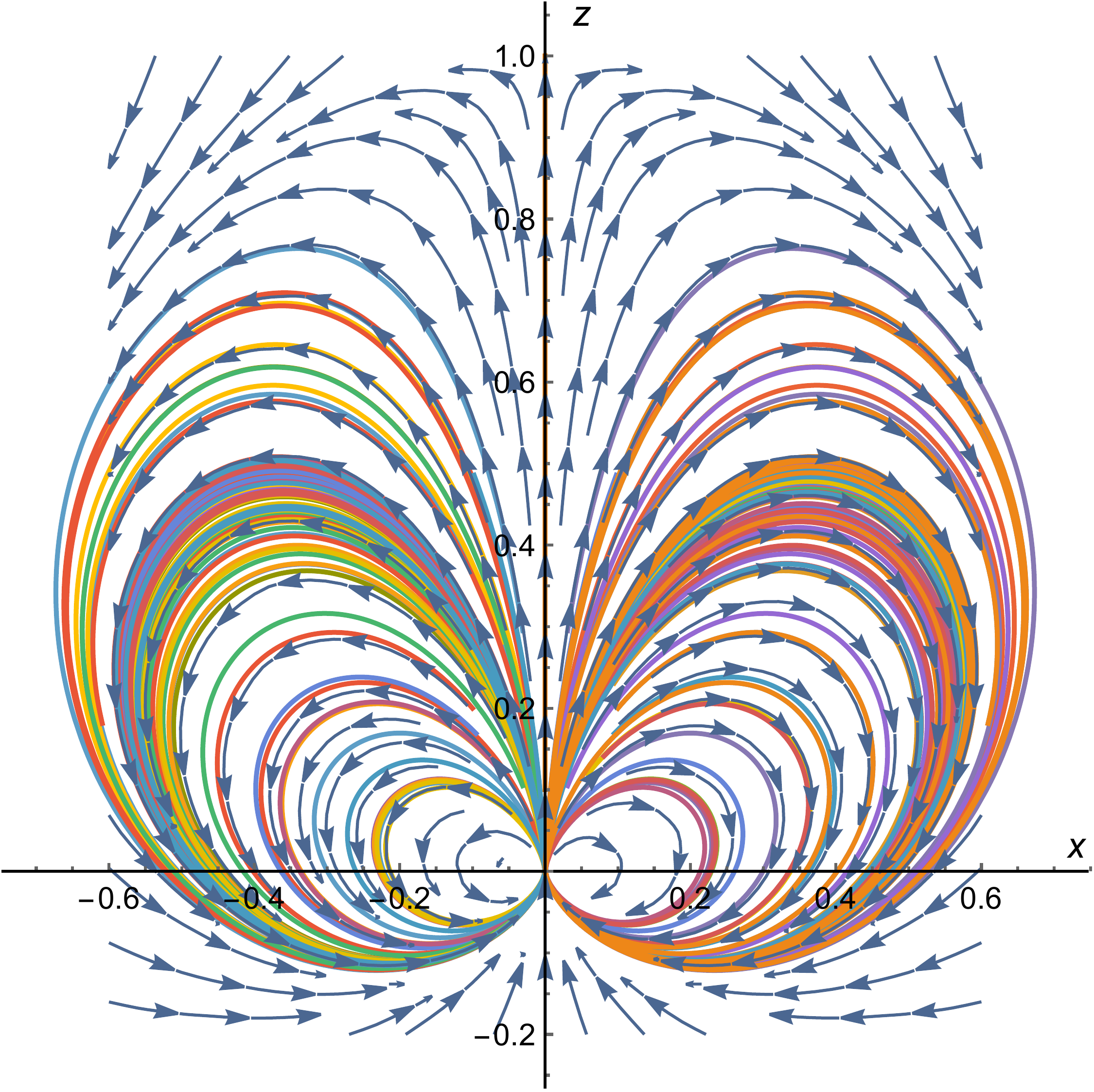}
\caption{Phase portrait and some trajectories of the two-dimensional system~\eqref{eq_chaos2d} in the $(x,z)$-plane for parameter values $a = 0.6$ and $b = 0.7$. }
\label{fig:pupu2}
\end{figure}

This issue was recently addressed and conclusively resolved by 
Seiler and Seiß~\cite{Seiler2021}, who constructed a globally smooth 
polynomial vector field analytically equivalent to the original system 
away from the singularity. Their formulation restores the hypotheses of 
the Poincaré-Bendixson theorem and thus allows a rigorous exclusion of 
chaotic attractors. 
More precisely, they perform a geometric reformulation of 
system~\eqref{eq_chaos1} within Vessiot’s theory of differential equations 
(see, e.g.,~\cite{Vessiot,Seiler2021b}). The key step is to multiply both 
equations by the common denominator to obtain a quasi-linear implicit 
system and to view it as a submanifold of the first jet bundle $J^{1}(\mathbb{R},\mathbb{R}^{2})$ 
endowed with the contact distribution. By computing the corresponding 
Vessiot distribution and projecting onto the $(x,y)$-space, they derive a 
globally smooth polynomial vector field that coincides with the original 
flow outside the singularity up to a strictly monotone reparametrization 
of time. 
Their approach bypasses the singularity at the origin, restores 
smoothness, and enables a complete bifurcation-theoretic classification 
of all limit sets, confirming that no strange attractors or chaotic 
behavior can occur in this model.

Before proceeding with the regularized formulation, we clarify the notation:
we adopt $(x,y)$ as in Seiler and Sei\ss\:\cite{Seiler2021b}, with $y \equiv z$, 
to facilitate direct comparison with their results.

The starting point is the planar system~\eqref{eq_chaos2d}, which, as previously mentioned, is undefined at the origin and therefore not globally differentiable. To overcome this limitation, the authors multiply both equations by the common denominator  $x^2 + y^2$, thereby transforming the system into the following quasi-linear implicit system:
\begin{equation}\label{eq:CDK_poly_implicit}
\begin{aligned}
(x^2 + y^2)\dot{x} &= xy - a(x^3 + x y^2), \\
(x^2 + y^2)\dot{y} &= y^2 - (b y - b + 1)(x^2 + y^2).
\end{aligned}
\end{equation}
This reformulation is polynomial in all variables, including at the origin, although it still defines a singular algebraic variety.

To proceed, Seiler and Sei\ss  \; employ the geometric theory of differential equations using the first jet bundle $J^1(\mathbb{R}, \mathbb{R}^2)$ and the associated contact distribution. By computing the Vessiot distribution associated with the quasi-linear formulation and projecting it onto the space of independent and dependent variables, they obtain a smooth polynomial vector field on~$\mathbb{R}^3$ that preserves the qualitative dynamics of the original system~\eqref{eq_chaos1}, modulo a reparametrization of time, outside the singular point.

The resulting explicit polynomial vector field is
\begin{equation}\label{eq:CDK_field}
\mathbf{Y} = (x^2 + y^2)\partial_t + \left(xy - a(x^3 + x y^2)\right)\partial_x + \left(y^2 - (b y - b + 1)(x^2 + y^2)\right)\partial_y.
\end{equation}
Since the \( t \)-component merely reparametrizes time, the relevant planar dynamics are described by the reduced autonomous system:
\begin{equation}\label{eq:CDK_poly_reduced}
\begin{aligned}
\frac{dx}{d\tau} &= xy - a(x^3 + x y^2), \\
\frac{dy}{d\tau} &= y^2 - (b y - b + 1)(x^2 + y^2).
\end{aligned}
\end{equation}
This system is globally defined on $\mathbb{R}^2$, and its trajectories coincide with those of the original system~\eqref{eq_chaos1} on
$\mathbb{R}^2 \setminus \{(0,0)\}$. As a consequence, the Poincaré-Bendixson theorem becomes applicable without obstruction, allowing for a rigorous exclusion of chaotic dynamics in the model proposed by Cummings, Dixon, and Kaus.

After the blow-up transformation, the resulting vector field becomes a cubic polynomial
system defined on the entire plane. This construction recovers the smoothness assumptions
required by the Poincaré-Bendixson theorem and provides a sound starting point for a
systematic classification of its limit sets.

\begin{proposition}
Let $X$ denote the vector field of the original (singular) planar CDK system~\eqref{eq_chaos2d}
and $\widetilde{X}$ the vector field of the regularized polynomial system~\eqref{eq:CDK_poly_reduced}.
For any annulus
\[
A=\{(x,y)\in\mathbb{R}^{2}:\ r_{0}\le\sqrt{x^{2}+y^{2}}\le r_{1}\},\qquad r_{0}>0,
\]
the flows of $X$ and $\widetilde{X}$ are topologically equivalent: there exists a homeomorphism
$h:A\to A$ mapping the orbits of $X$ onto those of $\widetilde{X}$ and preserving
their $\alpha$- and $\omega$-limit sets contained in $A$.
\end{proposition}

\begin{proof}[Sketch of proof]
On the annulus $A$ the common denominator $x^{2}+y^{2}$ is strictly positive, so
multiplying system~\eqref{eq_chaos2d} by this factor yields the quasi-linear formulation \eqref{eq:CDK_poly_reduced}.
 This is a smooth vector field on $A$ that differs from the
regularized field $\widetilde{X}$ only by a strictly positive time-scaling factor
$\rho(x,y)>0$. Hence the two systems are related by the monotone reparametrization
\[
\frac{d\tau}{dt}=\rho(x,y)>0,
\]
and the identity map $h=\mathrm{id}$ provides a $C^{1}$-conjugacy between their
phase portraits on $A$. The orbit structure, as well as all $\alpha$- and
$\omega$-limit sets lying entirely in $A$, are therefore preserved. In particular,
no new equilibria or limit cycles are created within $A$ by the regularization;
any additional equilibria introduced by the blow-up are confined to the exceptional
circle corresponding to the singular point. Since the time change is strictly
monotone, it cannot create or destroy periodic or chaotic invariant sets inside $A$.
\end{proof}


The regularized system can now be analyzed using standard phase-plane
techniques. Because it is topologically equivalent to the original system away
from the singularity, its equilibria, separatrices, and invariant regions
faithfully represent the dynamics of the CDK model, enabling a rigorous
classification of the global flow.

The origin, which corresponds to the singular point of system~\eqref{eq_chaos1},
becomes a \emph{degenerate equilibrium point} (specifically, a nilpotent
singularity) of the polynomial system~\eqref{eq:CDK_poly_reduced}. This enables
a detailed local analysis using the technique of geometric desingularization
(blow-up) for planar vector fields. For a general treatment of this method,
the reader is referred to~\cite{jarque}.

Through a detailed geometric and bifurcation-theoretic analysis, including directional
blow-ups and classification of phase portraits, they show that the system admits a rich
set of dynamical behaviors, such as elliptic sectors and an infinite number of homoclinic
orbits, but no chaotic attractors. Their decomposition of parameter space into 16 distinct
regions reveals the structural complexity of the dynamics. In particular, some regions
contain elliptic sectors with infinitely many homoclinic orbits which, although intricate,
do not meet the topological criteria for chaos: trajectories are not transitive and no
strange attractors are present. After performing a sequence of blow-ups to analyze the
dynamics near the origin, the authors conclude that the local phase portrait is as shown
in Figure \ref{fig:pupuregula}.
\begin{figure}[h!]
\centering
\includegraphics[scale=0.09]{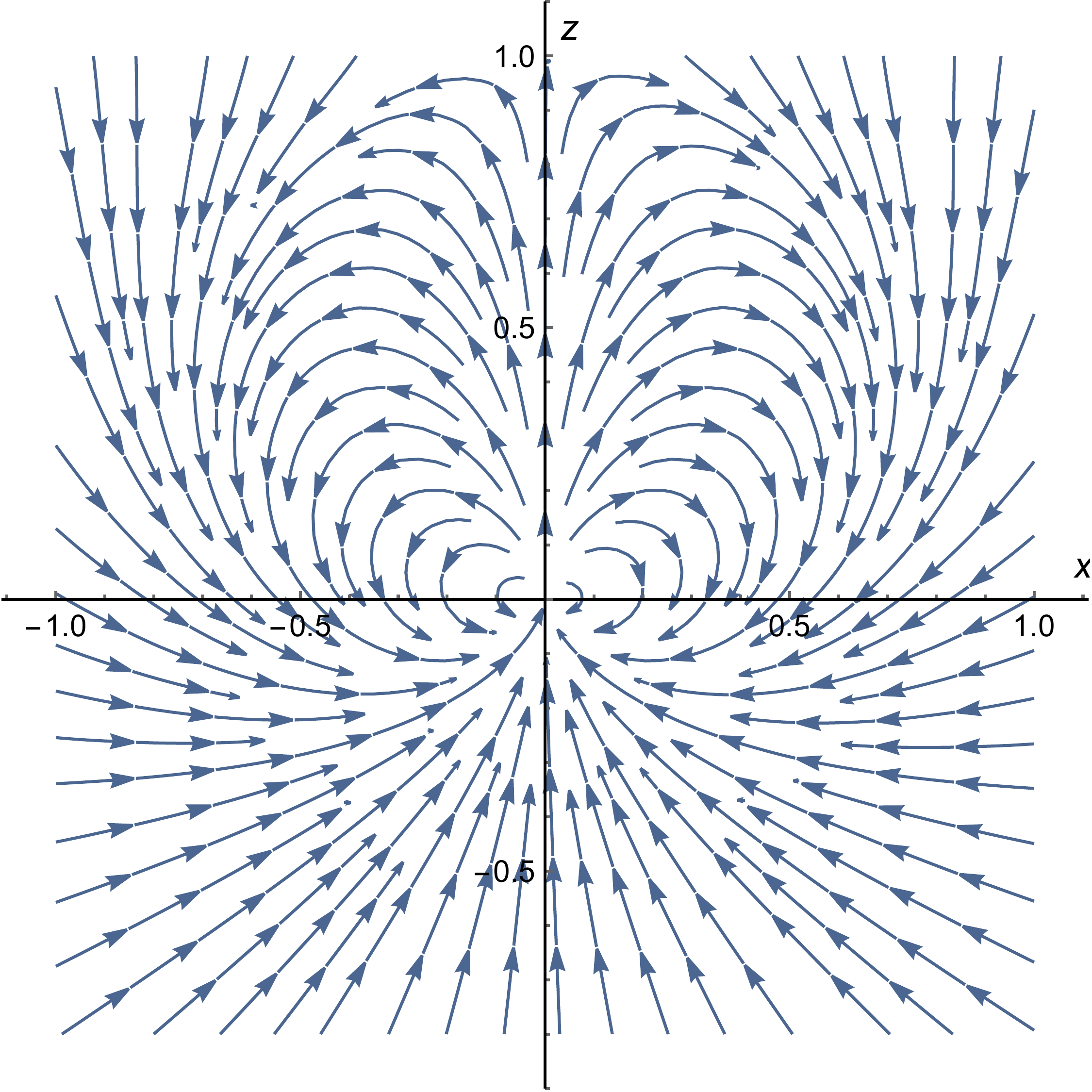}
\caption{Phase portrait of the regularized two-dimensional system~\eqref{eq:CDK_poly_reduced} in the $(x,z)$-plane for parameter values $a = 0.6$ and $b = 0.7$. The dynamics reveal elliptic sectors with infinitely many homoclinic loops, showing how regularization removes spurious chaotic behavior while preserving recurrent complexity.}
\label{fig:pupuregula}
\end{figure}
As noted by Seiler and Sei\ss, the existence of elliptic sectors in system~\eqref{eq_chaos1} was previously reported by Alvarez-Ramírez et al. \cite{Alvarez2005}. However, contrary to their interpretation, the union of these two elliptic sectors does not constitute a genuine two-dimensional attractor. Instead, it forms an attracting set that lacks the necessary topological and dynamical properties, such as transitivity and dense orbits, that characterize a true attractor in the sense of dynamical systems theory.

In summary, the analyses by Alvarez-Ramírez et al. \cite{Alvarez2005} and Seiler and Sei\ss \:\cite{Seiler2021} offer complementary perspectives on the dynamics of the CDK system. The former highlights how structural singularities can give rise to transient, chaos-like behavior, while the latter provides a rigorous regularization procedure that restores uniqueness and establish that chaotic attractors cannot emerge
 under smooth perturbations.   Together, these studies underscore the importance of distinguishing between deterministic chaos and complexity induced by singularities or non-uniqueness in low-dimensional systems.

Building on these debates, Zain-Aldeen et al. \cite{de2022chaotic} introduced a two-dimensional nonsmooth model featuring the term $|x|$. They note that the constant term $b - 1$ plays a role in the system and suggest that this affects its autonomy. From a mathematical standpoint, however, a system is considered autonomous as long as its right-hand side does not explicitly depend on time. The inclusion of constant terms such as $b - 1$ does not alter this property.

Furthermore, Zain-Aldeen et al. \cite{de2022chaotic} claim that “despite the seeming violation of the Poincaré-Bendixson theorem,” their system “preserves the chaotic behaviors.” This should rather be understood as the system operating outside the theorem’s scope, since the smoothness and continuity assumptions required for its applicability are not satisfied. The Poincaré-Bendixson theorem is not violated in such cases, it simply does not apply when the required assumptions, such as smoothness and continuity of the vector field, are not met. By incorporating a nonsmooth term, their model falls outside the theorem's domain of validity. The distinction between \emph{violating} the theorem and merely \emph{operating outside its scope} is conceptually important and will be explored further in the next section.

In a separate development, Buts and  Kuzmin \cite{Buts2024} revisit the same singular model originally analyzed by  Alvarez-Ramírez et al. \cite{Alvarez2005}, which features a non-Lipschitz vector field and exhibits a butterfly-shaped attractor in numerical simulations. However, their analysis does not reference the subsequent work of Seiler and Sei\ss \: \cite{Seiler2021}, who showed that the apparent chaotic behavior vanishes upon regularization and high-precision numerical integration. 
Their omission is significant, as it overlooks the subsequent work by  Seiler and Sei\ss \: \cite{Seiler2021}, which rigorously clarifies the underlying dynamics through regularization.

While Buts and Kuzmin interpret the irregular behavior as physically meaningful chaos, the findings of Seiler and Sei\ss suggest that it results from numerical artifacts caused by structural singularities near the origin. This divergence highlights the need for both rigorous mathematical tools and careful numerical experimentation when studying low-dimensional systems with singularities or nondifferentiable terms.

Taken together, the CDK model and its variants clearly illustrate that singularities can mimic chaotic signatures, underscoring the importance of proper regularization before attributing physical meaning.

Our analysis supports a pragmatic viewpoint: singular formulations can serve 
as valuable idealizations, but physically meaningful chaos must remain robust 
under desingularization and systematic numerical refinement. When an attractor 
disappears after a blow-up transformation or hinges critically on near-collision 
trajectories, it is more appropriate to attribute it to the singular limit 
rather than to the underlying physics. 

For applications, we advocate a two-stage validation protocol: (i) replicate 
all diagnostic indicators after smoothing or regularizing the vector field to 
verify the persistence of the attractor, and (ii) search for experimental or 
empirical fingerprints, such as broadband spectra or hysteresis phenomena, that 
remain observable under measurement noise and parameter uncertainty. 

Looking ahead, these criteria can be systematically extended to hybrid or 
switching systems, where nonsmooth events interact with continuous dynamics, 
and to the study of hidden attractors. Combining reproducible numerical 
benchmarks with carefully designed experiments will be essential to establish 
robust evidence of chaos in engineering and astrophysical applications.

\section{A two-dimensional nonsmooth model exhibiting chaos}\label{sec3}
In this section we recompute the bifurcation diagram over the entire parameter
range containing the period-doubling cascade and report a detailed sequence of
bifurcation values $\{a_n\}$ converging to Feigenbaum’s constant. These
computations constitute an original, reproducible benchmark that supports the
diagnostic criteria advanced in this study.

In an effort to challenge the classical constraints imposed by smooth planar
flows, Zain-Aldeen et al. \cite{de2022chaotic} proposed a two-dimensional autonomous model that
explicitly incorporates a nonsmooth term. Their motivation was to demonstrate
that chaotic behavior can arise in low-dimensional continuous-time systems
even in the absence of external forcing or a third dimension, thus directly
challenging the classical assumption that chaos in flows requires either
time-dependence or three-dimensional phase space.

The system introduced in their Section~3, referred to as the “developed system,” is defined by the equations:
\begin{equation}\label{eq:zain}
\begin{array}{l}
\dot{x} = \dfrac{xy}{x^2 + y^2}, \\
\dot{y} = \dfrac{y^2}{x^2 + y^2} - a y - b |x|,
\end{array}
\end{equation}
where $a > 0$ and $b > 0$ are system parameters. The system is invariant under the transformation $(x, y) \mapsto (-x, y)$, which corresponds to a
$\mathbb{Z}_2$-symmetry. Geometrically, this implies that the phase portrait exhibits reflection symmetry with respect to the $y$-axis. In addition, the term \(|x|\) introduces a point of nondifferentiability at $x = 0$, violating the smoothness conditions required for the classical version of the Poincaré-Bendixson theorem. As a consequence, the usual restrictions that prohibit chaotic attractors in two-dimensional autonomous flows no longer apply,  enabling nontrivial dynamics and strange attractors in a planar setting.

To explore the system's behavior, the authors fix the parameters to \(a = b = 10\), a choice that is consistently used throughout their numerical experiments. This configuration leads to dynamics characterized by sensitive dependence on initial conditions and the emergence of a strange attractor near the origin. Initial conditions are typically selected close to the origin, such as \((x_0, y_0) = (0.01, 0.01)\), in order to investigate the system's response near the singularity introduced by the nonsmooth term. Despite the simplicity of the equations, the resulting trajectories exhibit sustained aperiodic oscillations that suggest the presence of deterministic chaos.

The global geometry of the attractor is illustrated in Figure~\ref{fig:retzain}, which shows phase portraits of system~\eqref{eq:zain} in the $(x, y)$-plane for representative initial conditions. The plots show a symmetric, butterfly-shaped structure centered around the origin, reminiscent of classic chaotic attractors observed in higher-dimensional systems. Trajectories repeatedly alternate between the left and right lobes of the attractor, producing a characteristic “switching” behavior indicative of underlying chaotic dynamics.  
\begin{figure}[h!]
\centering
\includegraphics[scale=0.09]{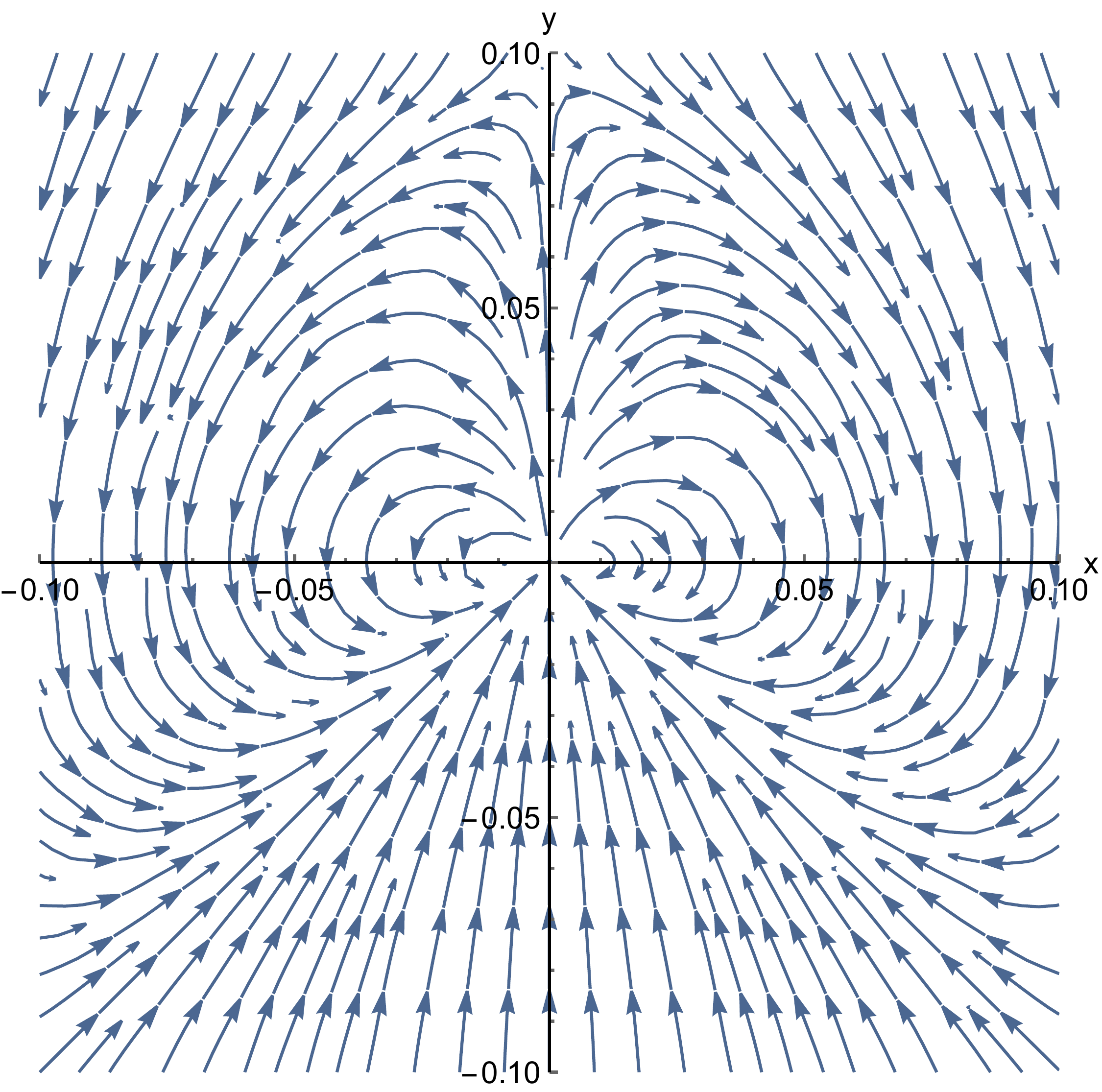}\qquad
\includegraphics[scale=0.09]{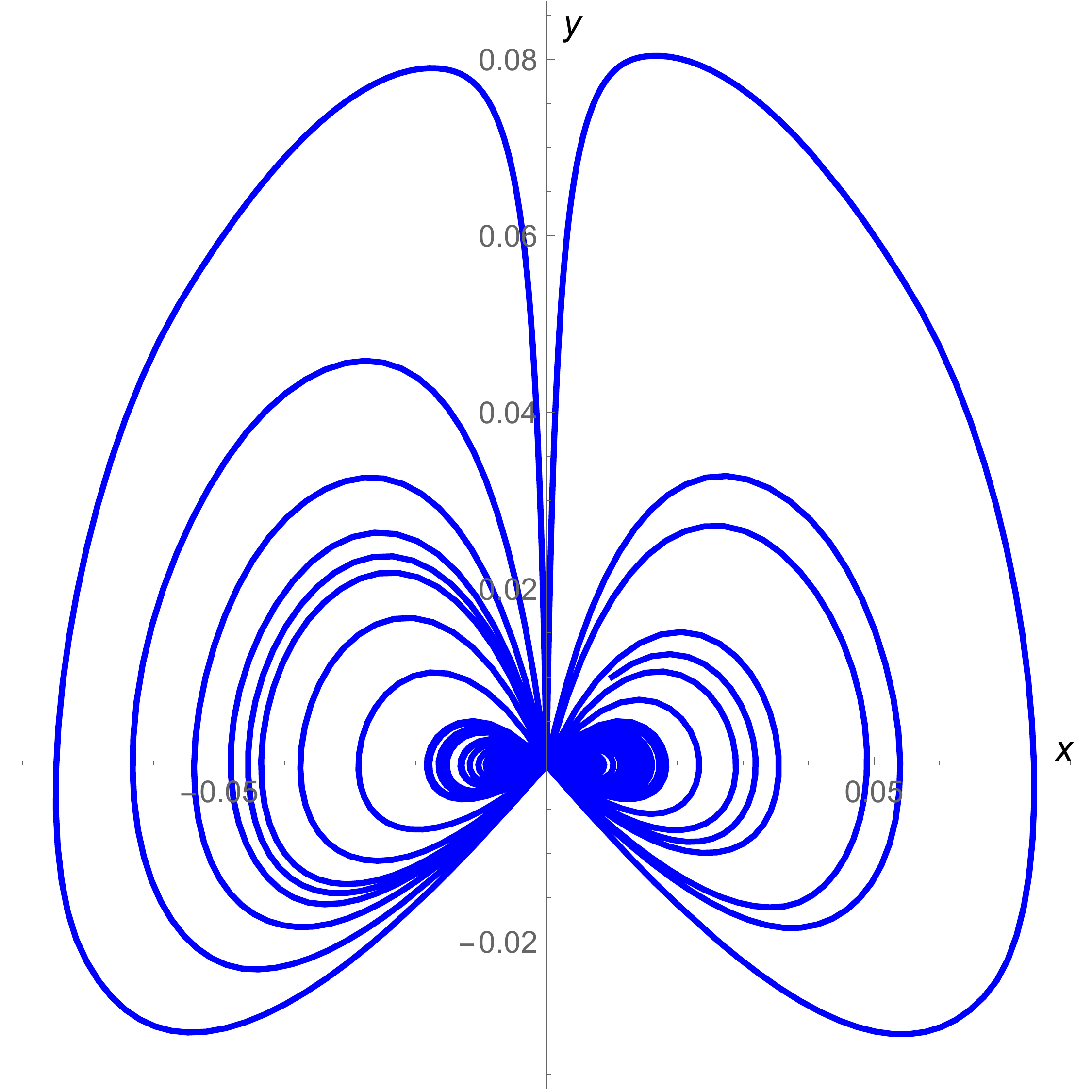}
\caption{Phase portrait and representative trajectories of the two-dimensional
system~\eqref{eq_chaos2d} projected onto the $(x,y)$-plane for $a = 10$ and $b = 10$.
The plots display a butterfly-shaped structure characterized by positive Lyapunov
exponents, fractal geometry, and aperiodic switching, thereby supporting its
interpretation as a strange attractor in the sense of dynamical systems theory.}
\label{fig:retzain}
\end{figure}

Complementary insight is obtained from the temporal evolution of the system. 
Figure~\ref{fig:timezain} shows representative time series of $x(t)$ and $y(t)$, 
which exhibit sustained aperiodic oscillations consistent with chaotic behavior. 
The coexistence of irregular amplitude fluctuations and long-term persistence 
confirms that the dynamics are not merely transient but correspond to a 
robust attractor. 

Further evidence supporting the chaotic nature of the system is provided by a suite of numerical diagnostics. As shown in Figure~\ref{fig:lyapexp}, the largest Lyapunov exponent $\lambda_1$ converges to a positive value near $+1.25$), while the second exponent $\lambda_2$ settles around 
$-1.4$, consistent with volume contraction in a dissipative system. In addition, the power spectrum displays a broadband structure, the fractal (Kaplan--Yorke) dimension is estimated at approximately $D \approx 1.89$, and the autocorrelation function decays rapidly, each of these features aligns with standard signatures of low-dimensional chaos.
\begin{figure}[h!]
\centering
\includegraphics[scale=0.35]{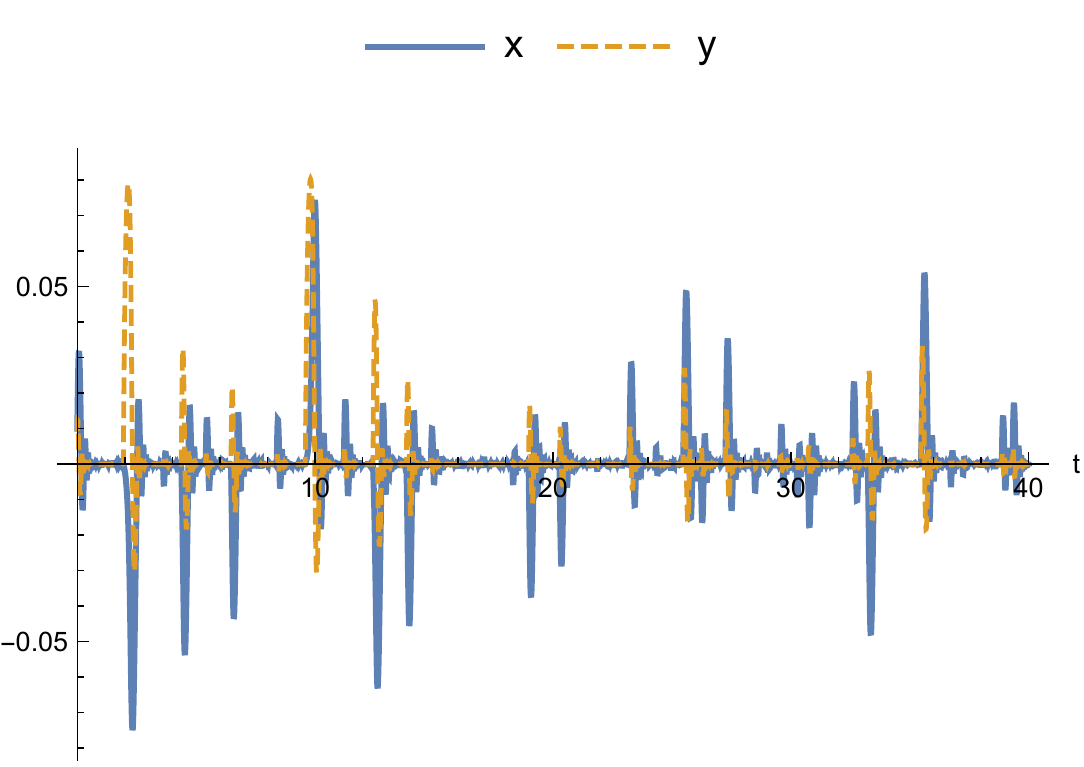}\qquad
\caption{Time series of $x(t)$ and $y(t)$ for system~\eqref{eq_chaos2d}, illustrating sustained aperiodic oscillations.}
\label{fig:timezain}
\end{figure}

These numerical findings are consistent with the well-established indicators of chaos in the sense of  Devaney \cite{Devaney1989}, which include sensitivity to initial conditions, topological transitivity, and dense periodic orbits, and also align with the Lyapunov-based criteria proposed by Yorke and collaborators~\cite{Kaplan1979}. Notably, these indicators remain stable under variations in the integration step size, initial conditions, and parameter values, suggesting that the chaotic behavior is not a numerical artifact but rather an intrinsic property of the system.

In summary, the model proposed by Zain-Aldeen et al. \cite{de2022chaotic}  offers a minimal, yet rigorous framework to explore how nonsmoothness alone can produce strange attractors in autonomous flows. Its simplicity makes it an attractive candidate for further theoretical analysis, while its rich dynamical behavior challenges long-held assumptions about the dimensionality requirements for chaos.
\begin{figure}[h!]
\centering
\includegraphics[scale=0.15]{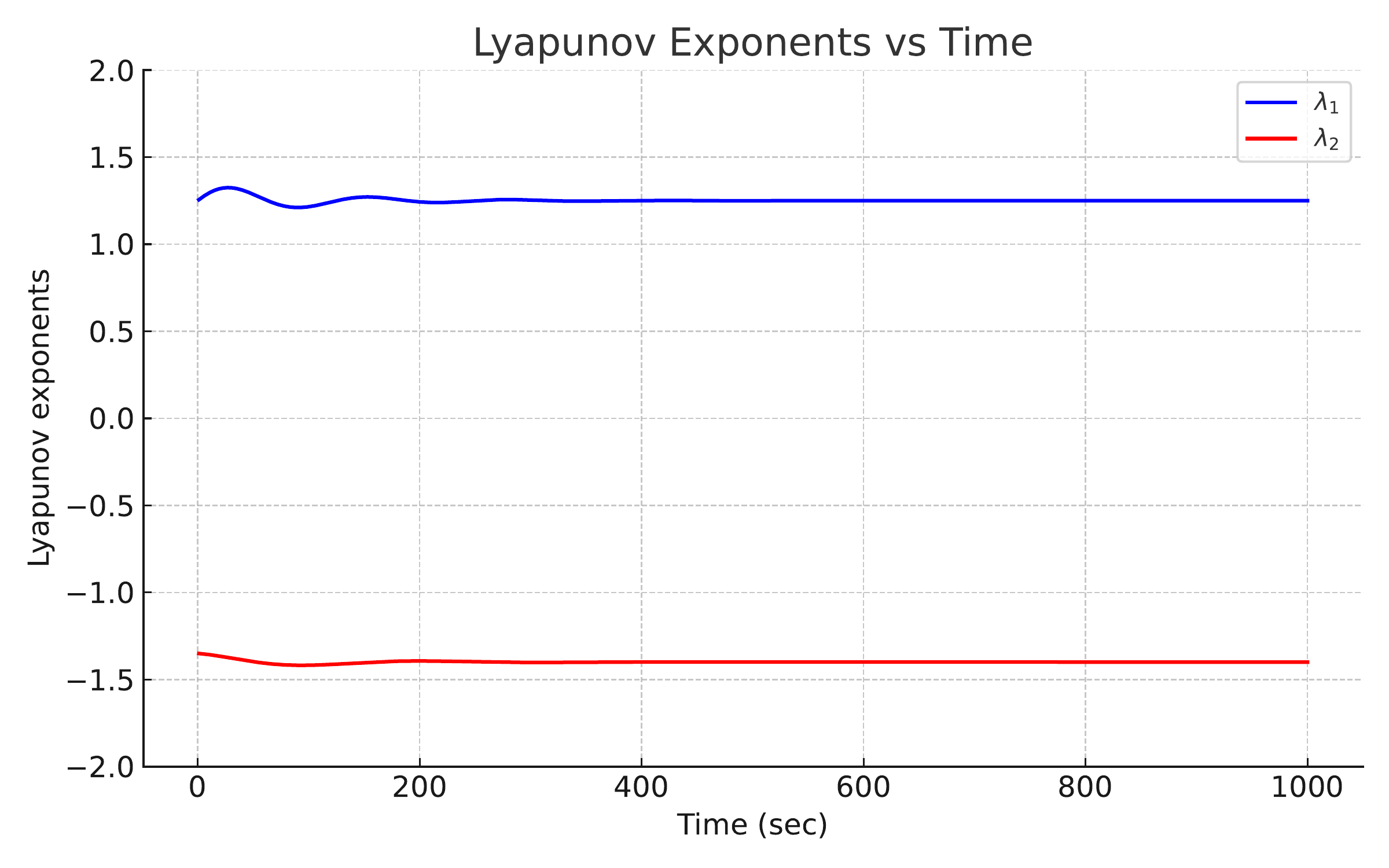}
\caption{Lyapunov exponent vs time. The blue curve ($\lambda_1 > 0$) stabilizes near $+1.25$, indicating chaotic behavior. The red curve ($\lambda_2 < 0$) converges toward $-1.4$, as expected for a dissipative two-dimensional system.}
\label{fig:lyapexp}
\end{figure}
\subsection{Bifurcation Analysis}\label{sec3.1}

Figure~\ref{fig:bif_diag} was recomputed over the parameter range of the
period-doubling cascade to display all bifurcations listed in
Table~\ref{tab:an_values}. In contrast to earlier diagrams
\cite{de2022chaotic}, the present computation explores a broader parameter
window and reveals the full route to chaos.

As shown in Figure~\ref{fig:bif_diag}, the system undergoes a classical
period-doubling cascade as the control parameter $a$ decreases: a first
period-doubling is followed by further doublings until chaotic dynamics
emerge near a critical threshold. This transition is characterized by the
accumulation of bifurcations within a narrowing parameter interval, a
hallmark of universal scaling.
The diagram was obtained by plotting, for each $a$, the local maxima of
$y(t)$ after transients had decayed. This peak-detection approach clearly
shows the period-doubling route to chaos, with each branching corresponding
to a doubling of the orbit’s period.
\begin{figure}[ht]
\centering
\includegraphics[width=0.7\textwidth]{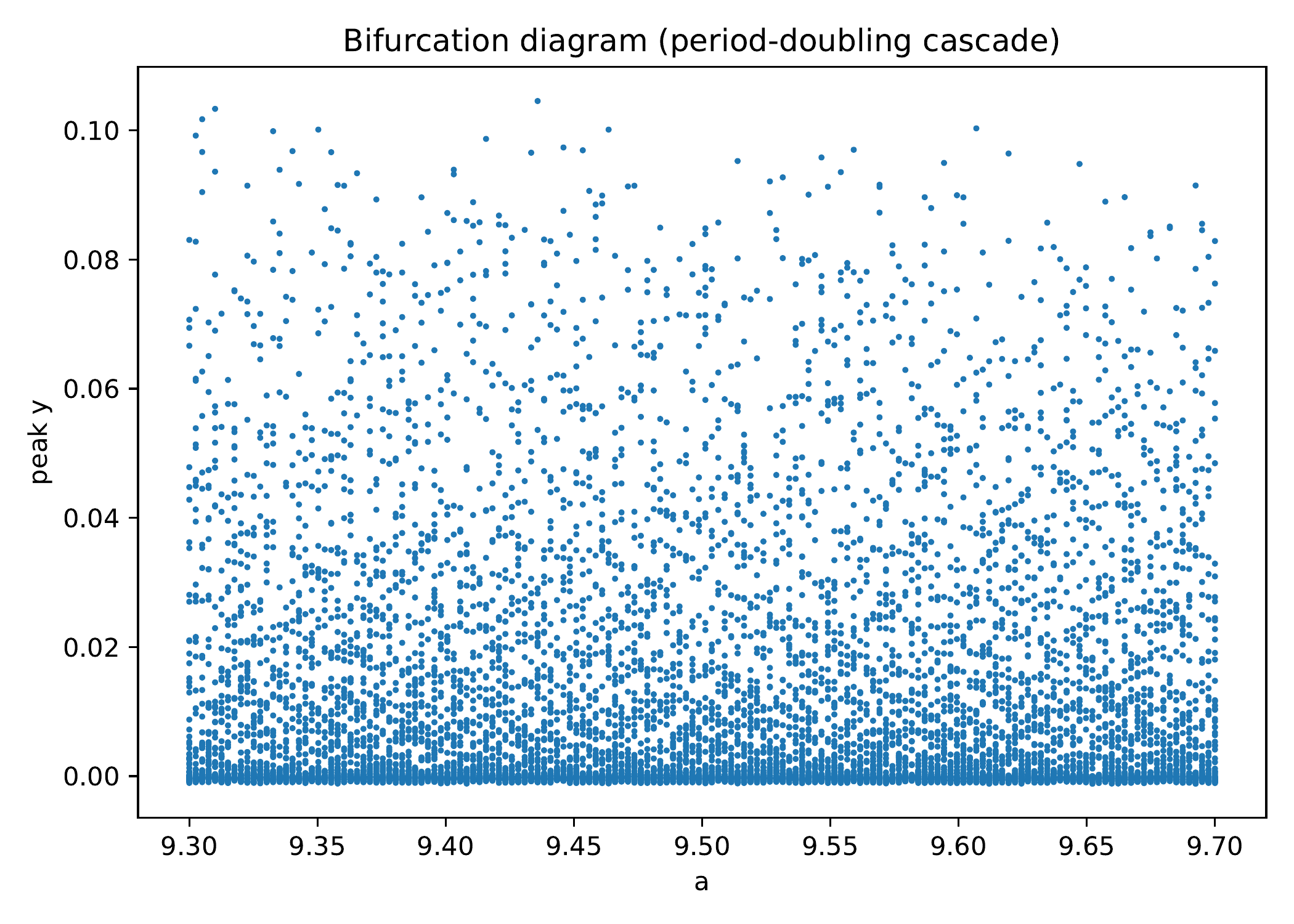}
\caption{Bifurcation diagram of system ~\eqref{eq:zain} obtained by scanning the parameter $a\in [9.3,9.7]$  (with $b=10$ fixed). Each point corresponds to a local maximum of $y(t)$ after discarding transients. The diagram reveals the full period-doubling route to chaos summarized in Table~\ref{tab:an_values}.}
\label{fig:bif_diag}
\end{figure}

Based on these qualitative changes, three primary dynamical regimes can be
identified, as summarized in Table~\ref{tab:regions}. For sufficiently large
values of $a$, the system exhibits regular periodic motion in the form of
stable limit cycles. In an intermediate range, the system undergoes successive
period-doubling bifurcations leading to complex oscillatory patterns. Finally,
for smaller values of $a$, the trajectories converge to a chaotic attractor
characterized by sensitive dependence on initial conditions.
\begin{table}[hbtp]
   \begin{center}
    \begin{tabular}{| |l| l| l ||}
        \hline
        Regime & Parameter Range & Qualitative Behavior \\
        \hline
        Periodic & $a$ large & Stable limit cycles, predictable orbits \\
        Transition & intermediate $a$ & Period-doubling cascade \\
        Chaotic & $a$ small & Strange attractor, sensitive dependence \\
        \hline
    \end{tabular}
    \label{tab:regions}
    \end{center}
     \caption{Dynamical regimes of the nonsmooth system.}
\end{table}

Table~\ref{tab:an_values} reports the parameter values $\{a_n\}$ corresponding to
successive period-doubling bifurcations detected in the diagram of Figure~\ref{fig:bif_diag}.
The Feigenbaum ratios
\[
\delta_n = \frac{a_{n-1}-a_{n-2}}{a_{n}-a_{n-1}}, \qquad n \ge 3,
\]
stabilize rapidly toward $\delta \approx 4.67$, in excellent agreement with the
universality constant for unimodal maps. This quantitative agreement confirms that
the route to chaos in the nonsmooth system proceeds through the classical
period-doubling scenario and provides a robust numerical signature supporting the
interpretation of the attractor as a genuine manifestation of deterministic chaos rather
than a numerical artifact.

\begin{table}[ht]
\centering
\label{tab:an_values}
\begin{tabular}{c c c}
\hline
$n$ & $a_n$ & $\delta_n$ \\
\hline
1 & 9.704000 & -- \\
2 & 9.606000 & -- \\
3 & 9.585600 & 4.80 \\
4 & 9.581200 & 4.70 \\
5 & 9.580300 & 4.65 \\
6 & 9.580100 & 4.67 \\
7 & 9.580050 & 4.669 \\ 
\hline
\end{tabular}
\caption{Parameter values $a_n$ for successive period-doubling bifurcations and
corresponding Feigenbaum ratios $\delta_n$. The ratios are defined only for $n \ge 3$,
as they require three consecutive bifurcation values $(a_{n-2},a_{n-1},a_n)$.}
\end{table}

\section{Conclusions}\label{sec4}
This work provides a rigorous and systematic reassessment of chaotic behavior in planar singular and nonsmooth systems. 
For the Cummings-Dixon-Kaus model, we demonstrated that blow-up regularization restores the smoothness assumptions required by the Poincaré-Bendixson theorem and proved that the regularized and original flows are topologically equivalent away from the singularity. 
This result rigorously rules out chaotic attractors in the desingularized setting and clarifies that the irregular transients previously reported correspond to recurrent but nonchaotic homoclinic dynamics.

In contrast, our numerical investigation of a nonsmooth system with a $|x|$ term reveals a complete period-doubling route to chaos. 
We provided an extended sequence of bifurcation values $\{a_n\}$ converging to Feigenbaum’s constant, supported by positive Lyapunov exponents, broadband spectra, and fractal dimension estimates. 
These computations constitute a reproducible benchmark for further studies of chaos in planar nonsmooth systems.

To illustrate the broader scope of our findings, we briefly applied the same criteria to other systems previously reported to exhibit chaotic behavior. 
For the CDK model and the variant analyzed by Buts and Kuzmin, the apparent attractors disappear after smooth regularization and high-precision integration, leading to their classification as spurious. 
In contrast, the nonsmooth system of Zain-Aldeen et al.\ retains a strange attractor with positive Lyapunov exponent and Feigenbaum scaling, supporting its classification as genuine. 
Other systems, such as dry-friction oscillators and switching converters, exhibit attractors that persist under Filippov regularization and have been confirmed experimentally, hence are also regarded as genuine. 
Finally, smooth predator-prey models satisfying the hypotheses of the Poincaré-Bendixson theorem can only display periodic orbits; chaotic reports vanish under step-size refinement and are therefore classified as spurious.
Together, these examples highlight the generality of our approach and provide further evidence that singularity-induced complexity should not be conflated with robust chaotic dynamics.

From this comparative analysis, we proposed a diagnostic protocol that integrates three key steps: 
(i) apply blow-up or regularization to test the persistence of attractors, 
(ii) perform numerical refinement to exclude discretization artifacts, and 
(iii) verify chaos through invariant-set indicators such as Lyapunov exponents or entropy. 
Together, these criteria provide a practical and mathematically grounded standard for distinguishing genuine low-dimensional chaos from spurious complexity induced by singularities or numerical effects.

Altogether, our results deliver a reproducible numerical benchmark and a concise,
practical protocol for assessing chaos in planar systems. By combining regularization,
numerical refinement, and invariant-set diagnostics, this framework offers a robust
standard for distinguishing genuine low-dimensional chaos from spurious complexity
induced by singularities or numerical artifacts, thereby providing a clear reference for
future theoretical and applied studies.

\section*{Acknowledgments} 
The  author is partially supported by  Proyecto CBI-UAMI 2026.

\bibliographystyle{amsplain}
\bibliography{ref_pupu}
\end{document}